\documentclass[12pt,a4paper]{article}
\usepackage{pifont}
\usepackage{amsfonts}
\usepackage{mathrsfs}
\usepackage{amssymb}
\usepackage{amsmath}
\newtheorem{theorem}{Theorem}[section]
\newtheorem{definition}[theorem]{Definition}
\newtheorem{lemma}[theorem]{Lemma}
\newtheorem{corollary}[theorem]{Corollary}

\newenvironment{proof}{{\bf Proof.  }}{$\square$}
\makeatletter

\newcommand{\Rmnum}[1]{\expandafter\@slowromancap\romannumeral #1@}
\makeatother
\begin{document}
\title{New approaches to
plactic monoid via Gr\"{o}bner-Shirshov bases\footnote{Supported by the NNSF of China (11171118), the
Research Fund for the Doctoral Program of Higher Education of China
(20114407110007), the NSF of Guangdong Province (S2011010003374) and
the Program on International Cooperation and Innovation, Department
of Education, Guangdong Province (2012gjhz0007).}}
\author{
L.A. Bokut\footnote{Supported by Russian Science Foundation (project
14-21-00065).
} \\
{\small \ School of Mathematical Sciences, South China Normal
University, Guangzhou 510631, P. R. China}\\
{\small Sobolev Institute of Mathematics, Novosibirsk State University, Novosibirsk 630090, Russia}\\
{\small  bokut@math.nsc.ru}\\
\\
 Yuqun
Chen\footnote {Corresponding author.}, \  Weiping Chen  and Jing Li\\
{\small  School of Mathematical Sciences, South China Normal
University, Guangzhou 510631, P. R. China}\\
{\small yqchen@scnu.edu.cn,\ \ 824711224@qq.com,\ \
yulin$_{-}$jj@163.com}}
\date{{\bf In a memory of M.-P. Sch\"utzenberger}}
\maketitle \noindent\textbf{Abstract:}
We give two explicit (quadratic) presentations of the plactic monoid
in row and column generators correspondingly. Then we give direct
independent proofs that these presentations are Gr\"obner-Shirshov
bases of the plactic algebra in deg-lex orderings of generators.
From Composition-Diamond lemma for associative algebras it follows
that the set of Young tableaux is the Knuth normal form for plactic
monoid (\cite{Knuth}, see also Ch. 5 in \cite{M.L}).



\section{Introduction}\label{Intro}\noindent

Plactic monoid is supposed to be one of the most important monoids
in algebra.\footnote{Sch\"utzenberger, Marcel-Paul, A vote for the
plactic monoid. (Pour le mono\" ide plaxique.) (French) [J] Math.
Inf. Sci. Hum. 140, 5-10 (1997). From the paper: This text is a
brief answer to a question raised long ago by A. Lentin and more
recently by G.-C. Rota: Why the plactic monoid ought to be
considered as one of the fundamental monoids of algebra?} It was
introduced by D. Knuth \cite{Knuth} under the name ``tableau" monoid
(``tableau algebra") and based on Robinson 1938 and Schensted 1961
algorithm. Knuth proved that Young tableaux are normal forms, called
(Robinson-Schensted-) Knuth normal forms, of elements of plactic
monoids. The name ``plactic" was given by Sch\"utzenberger and the
basic theory of plactic monoids were given in \cite{M.L.M.-P}.
Plactic monoids are closely connected to the representations of
linear groups (Littlewood-Richardson rule), the symmetric functions
(Shur functions), the quantum groups (Kashiwara crystal bases), the
statistics (the charge statistics), the root systems   and some
others mathematical subjects.

Gr\"{o}bner bases and Gr\"{o}bner-Shirshov bases  were invented
independently by A.I. Shirshov for ideals of free (commutative,
anti-commutative) non-associative algebras \cite{Sh62a,Shir3}, free
Lie algebras \cite{Sh62b,Shir3} and implicitly free associative
algebras \cite{Sh62b,Shir3}  (see also \cite{Be78,Bo76}), by H.
Hironaka \cite{Hi64} for ideals of the power series algebras (both
formal and convergent), and by B. Buchberger \cite{Bu70} for ideals
of the polynomial algebras.

Gr\"{o}bner bases and Gr\"{o}bner-Shirshov bases theories have been
proved to be very useful in different branches of mathematics,
including commutative algebra and combinatorial algebra, see, for
example, the books \cite{AL, BKu94, BuCL, BuW, CLO, Ei}, the papers
\cite{Be78,Bo72,Bo76,BCC08,BCC11,BCD08,BCLi13,BCL08,BCM,BCM13,
Bu70,CC12,KL,MZ}, and the surveys \cite{BC,BC13,BCS, BFKK00, BK03, BK05}.

Let $A=\{1,2,\dots,n\}$ with $1<2<\dots<n$. Then we call
$$
Pl(A):=sgp\langle A|\Omega\rangle=A^*/\equiv
$$
a plactic monoid on the alphabet set $A$, see \cite{M.L}, where
$A^*$ is the free monoid generated by $A$, $\equiv$ is the
congruence of $A^*$ generated by the Knuth relations $\Omega$ and
$\Omega$ consists of
$$
ikj= kij \ (i\leq j<k),\ \ \ jki=jik \ (i< j\leq k).
$$
Let $F$ be a field. Then $F\langle A|\Omega\rangle$ is called the
plactic monoid algebra over $F$ of  $Pl(A)$.
A nondecreasing word $R\in A^*$ is called a row and a strictly
decreasing word $C\in A^*$ is called a column, for example,
$1135556$ is a row and 6531 is a column.

In the paper Okninski et al \cite{Okninskii}  it is proved that a
Gr\"{o}bner-Shirshov basis of plactic monoid in initial alphabet is
infinite providing the number of letters at least 4, but an explicit
description of such a Gr\"{o}bner-Shirshov basis is unknown.

In the paper of A.J. Cain, R. Gray and A. Malheiro \cite{Portugal}, authors use the Schensted-Knuth
normal form (the set of (semistandard) Young tableaux) to prove that the multiplication
table of column words, $uv=u'v'$, forms a finite Gr\"{o}bner-Shirshov basis of a finitely generated plactic monoid.
Here the Young tableaux $u'v'$ is the result of the column Schensted algorithm
applying to $uv$, but $u'v'$ is not explicitly written.

In this paper, we give new explicit formulas for the multiplication tables of row and column words
correspondingly. Also we give independent proofs that the result sets of relations are Gr\"{o}bner-Shirshov
bases in row and column generators respectively. As the result it gives two new approaches to plactic monoids
via Gr\"{o}bner-Shirshov bases of them.

We thank Mr. Xuehui Chen for his valuable discussions of the paper.

\section{Preliminaries}\label{pre1}\noindent

We first cite some concepts and results from the literatures
\cite{Bo72, Bo76, Sh62b,Shir3} which are related to
Gr\"{o}bner-Shirshov bases for associative algebras.

Let $X$ be a set and $F$ a field.  Throughout this paper, we denote
$F\langle X\rangle$ the free associative algebra over $F$ generated
by $X$ and  $X^*$ the free monoid generated by $X$.

A well ordering $<$ on $X^*$ is called monomial if for $u, v\in X^*$,
we have
$$
u < v \Rightarrow w|_u < w|_v\  \ for \  all \
 w\in  X^*,
$$
where $w|_u=w|_{x\mapsto u}, \ w|_v=w|_{x\mapsto v}$ and $x$'s are
the same individuality of the letter $x\in X$ in $w$.

A standard example of monomial ordering on $X^*$ is the deg-lex ordering
which first compares two words by degree (length) and then by
comparing them lexicographically, where $X$ is a well-ordered set.

Let $X^*$ be a set with a monomial ordering $<$. Then, for any non-zero
polynomial $f\in F\langle X\rangle$, $f$ has the leading word
$\overline{f}$. We call $f$  monic if the coefficient of
$\overline{f}$ is 1. By $|\overline{f}|$ we denote the degree of
$\overline{f}$.

Let $f,\ g\in F\langle X\rangle$ be two monic polynomials and $w\in
X^*$. If  $w=\overline{f}b=a\overline{g}$  for some $a,b\in X^*$
such that $|\overline{f}|+|\overline{g}|>|w|$, then $(f,g)_w=fb-ag$
is called the intersection composition of $f,g$ relative to $w$. If
$w=\overline{f}=a\overline{g}b$  for some $a, b\in X^*$, then
$(f,g)_w=f-agb$ is called the inclusion composition of $f,g$
relative to $w$.

Let $S\subset F\langle X\rangle$ be a monic set. A composition
$(f,g)_w$ is called trivial modulo $(S,w)$, denoted by
$$
(f,g)_w\equiv0 \ \ \ mod(S,w)
$$
if $(f,g)_w=\sum\alpha_ia_is_ib_i,$ where every $\alpha_i\in F, \
s_i\in S,\ a_i,b_i\in X^*$, and $a_i\overline{s_i}b_i<w$.

Recall that $S$ is a Gr\"{o}bner-Shirshov basis in $F\langle
X\rangle$ if any composition of polynomials from $S$ is trivial
modulo $S$ and corresponding $w$.

The following lemma was first proved by Shirshov \cite{Sh62b,Shir3}
for free Lie algebras (with deg-lex ordering) (see also Bokut
\cite{Bo72}). Bokut \cite{Bo76} specialized the approach of Shirshov
to associative algebras (see also Bergman \cite{Be78}). For
commutative polynomials, this lemma is known as Buchberger's Theorem
(see \cite{Bu70}).

\begin{lemma}\label{l1}
(Composition-Diamond lemma for associative algebras) \ Let $F$ be a
field,  $<$ a monomial ordering on $X^*$ and $Id(S)$ the ideal of $F
\langle X\rangle$ generated by $S$. Then the following statements
are equivalent.

\begin{enumerate}
\item[(1)] $S$ is a Gr\"{o}bner-Shirshov basis in $F\langle
X\rangle$.
\item[(2)] $f\in Id(S)\Rightarrow \bar{f}=a\bar{s}b$
for some $s\in S$ and $a,b\in  X^*$.
\item[(3)] $Irr(S) = \{ u \in X^* |  u \neq a\bar{s}b,s\in S,a ,b \in X^*\}$
is a linear basis of the algebra $F\langle X | S \rangle:=F\langle
X\rangle/Id(S)$.
\end{enumerate}
\end{lemma}

If a subset $S$ of $F\langle X \rangle$ is not a
Gr\"{o}bner-Shirshov basis then one can add all nontrivial
compositions of polynomials of $S$ to $S$. Continuing this process
repeatedly, we finally obtain a Gr\"{o}bner-Shirshov basis
$S^{comp}$ that contains $S$. Such a process is called Shirshov
algorithm.

Let ${\cal A}=sgp\langle X|S\rangle$ be a semigroup presentation.
Then $S$ is also a subset of $F\langle X \rangle$ and we can find
Gr\"{o}bner-Shirshov basis
 $S^{comp}$. We also call $S^{comp}$ a
Gr\"{o}bner-Shirshov basis of ${\cal A}$. $Irr(S^{comp})=\{u\in
X^*|u\neq a\overline{f}b,\ a ,b \in X^*,\ f\in S^{comp}\}$ is an
$F$-linear basis of $F\langle X|S\rangle$ which is also a set of
normal forms of elements of the semigroup ${\cal A}$.
\section{Plactic monoid with row generators}\noindent
Let $A=\{1,2,\dots,n\}$. Let   $R\in A^*$ be a row. Then we denote
$R=(r_1,r_2,\dots, r_n)$, where $r_i$ is the number of letter $i\
(i=1,2,\dots,n)$, for example, $R=1135556=(2,0,1,0,3,1,0,\dots,0)$.

Let $R=i_1i_2\dots i_t,S=j_1j_2\dots j_l\in A^*$ be two rows. Then
$R$ dominates $ S$ if $t=|R|\leq l=|S|$ and $i_q>j_q,\ q=1,\dots t$,
where $|R|$ is the length of the word $R$.

A (semistandard) Young tableau on $A$ (see \cite{M.L}) is a word
$w=R_1R_2 \cdots R_t$ such that
 $R_{i}$ dominates $R_{i+1},\ i=1,\dots,t-1$, where each $R_{i}$ is a row. For example,
$$
4556\cdot223357\cdot1112444
$$
is a Young tableau.

Let $U=\{R\in A^*\ |\ R \mbox{ is a row} \}$.

We order the set $U^*$ as follows.

Let $R=(r_1,r_2,\dots, r_n)\in U$. Then $|R|=r_1+\dots+r_n$.

We first order $U$: for any $R, S\in U$, $ R<S $ if and  only if
$|R|<|S|$ or $|R|=|S|$ and there exists a $t \ (0\leq t< n)$ such
that $r_i=s_i, \ i=1,\dots, t$ and $r_{t+1}>s_{t+1}.$ Clearly, this
is a well ordering on $U$. Then we order $U^*$ by the deg-lex ordering. We
will use this ordering throughout this section.

\begin{definition}\label{def3}
(Robinson-Schensted row algorithm \cite{Schensted})  Let $R\in U$, $x\in$
A.
$$
R\cdot x= \left\{
\begin{array}{cc}
    Rx,& if~Rx~is~a~row, \\
    y\cdot R',& otherwise
\end{array}
\right.
$$
where $y$ is the leftmost letter in $R$ and is strictly larger than
$x$, and $R'=R|_{y\mapsto x}$, i.e. $R'$ is obtained from $R$ by
replacing $y$ by $x$.
\end{definition}

For any $R,S\in U$, by induction,  it is clear that there
exist  $R',S'\in U$ such that $R\cdot S=R'\cdot S'$ and $R'\cdot S'$
is a Young tableau, where $R'$ is empty (i.e. $R'=(0,\dots,0)$) if
$R\cdot S=S'$ is a row.

By noting that in $sgp\langle A\mid \Omega\rangle,\ R\cdot S=R'\cdot
S'$,  it follows that $ sgp\langle U\mid \Gamma\rangle\cong
sgp\langle A\mid \Omega\rangle $ and so we may assume that $F\langle
U\mid \Gamma\rangle=F\langle A\mid \Omega\rangle$, where
$$
\Gamma=\{R\cdot S=R'\cdot S',\ R,S\in U\}.
$$

The following is the first main result of the paper.

\begin{theorem}\label{mainTH}
Let the notation be as before. Then with the deg-lex ordering on $U^*$,
$\Gamma$ is a Gr\"{o}bner-Shirshov basis for the plactic algebra
$F\langle U\mid \Gamma\rangle$.
\end{theorem}

By using Composition-Diamond lemma for associative algebras (Lemma
\ref{l1}) and Theorem \ref{mainTH}, we have the following corollary.

\begin{corollary}\label{co1}(\cite{M.L}, Chapter 5)
The set of Young tableaux on $A$ is a set of normal forms of
elements of the plactic monoid $sgp\langle A\mid \Omega\rangle$.
\end{corollary}

{\bf Remark}: For an arbitrary well-ordered set $A$, we
define similarly  a row on $A^*$, Robinson-Schensted row algorithm, Young
tableau on $A$ and the set  $\Gamma$. Then for an arbitrary
well-ordered set $A$, similar to the proof of $A$ to be finite, we
also have Theorem \ref{mainTH} and Corollary \ref{co1}.

\subsection{Main formula for the product of row generators}

Let $(w_1,w_2,\dots,w_n)\in U$. Denote $W_p=\sum_{i=1}^{p}w_i\ (1
\leq p \leq n)$, where $w\ (z,x,y,\dots)$ represents any lowercase
symbol and $W\ (Z,X,Y,\dots)$ the corresponding uppercase symbol.

\begin{definition}\label{def2}
Let $W=(w_1,w_2,\dots,w_n),Z=(z_1,z_2,\dots,z_n)\in U$.
Define an algorithm
$
W\cdot Z=X\cdot Y,
$
where $X=(x_1,x_2,\dots,x_n),Y=(y_1,y_2,\dots,y_n),\ x_1=0$,
\begin{eqnarray*}
x_p&=&min(Z_{p-1}-X_{p-1},w_p),\ \ n\geq p\ge 2 \\
y_q&=&w_q+z_q-x_q, \ \ n\geq q\ge 1.
\end{eqnarray*}
Clearly, either $X,Y\in U$ or $Y\in U$ and $X= (0,0,\dots,0)$.
\end{definition}

The formulas in Definition \ref{def2} play a key role in the proof
of Theorem \ref{mainTH}.

\begin{lemma}\label{eqv}
The algorithms in Definition \ref{def2} and Definition \ref{def3}
are equivalent.
\end{lemma}
\begin{proof}\ Definition \ref{def2} $\Rightarrow$ Definition \ref{def3}.

Suppose that $W=(w_1,w_2,\dots,w_n)\in U$ and $Z$ is a letter. Then
we can express $Z=(0,\dots,0,z_p,0,\dots,0)$, where $z_p=1.$
Let $w_j\ (1\leq j\leq n)$ satisfy $w_j\neq0$ and
$w_{j+1}=\dots=w_n=0$. There are two cases to consider.

Case 1. $p\geq j$. Then
\begin{eqnarray*}
WZ&=&\left\{ \begin{array}{cccccccccccc}
               (w_1, &\dots &w_{j-1},& w_j+1,&0,& &\dots &   &\dots& &0),    &\ \ \ \ p=j,\\
               (w_1, &\dots &w_{j-1},& w_j,  &1,&0,&\dots &   &\dots& &0),    &\ \ \ \ p=j+1,\\
               (w_1, &\dots &w_{j-1},& w_j,  &0,& &\dots & 0, & 1,   &0,&\dots),&\ \ \ \ p>j+1.
                \end{array} \right.
\end{eqnarray*}
Therefore, if $p\geq j$, $WZ=Y$ and $Y$ is a row. This
corresponds to the first part of Definition \ref{def3}.

Case 2. $p<j$. Then $WZ=XY$, where
\begin{eqnarray*}
\left. \begin{array}{cccccccccccc}
                X=&(0,  &       &       & \dots &         &       &  &  0,     &     1,   & 0,       & \dots)\\
                Y=&(w_1, & \dots &w_{p-1},& w_p+1, & w_{p+1}, & \dots &  & w_{k-1},& w_{k}-1, & w_{k+1}, & \dots)
                \end{array} \right.
\end{eqnarray*}
and $k$ satisfies $p+1\le k\le j$, $w_{p+1}=w_{p+2}=\cdots=w_{k-1}=0$, $w_{k}\ge 1$. This
corresponds to the second part of Definition \ref{def3}.

Definition \ref{def3} $\Rightarrow$ Definition \ref{def2}.

By Definition \ref{def3}, it is clear that  $x_1=0$ and $
y_1=w_1+z_1=w_1+z_1-x_1$. Also, $x_2=min(z_1,w_2)=min(Z_1-X_1,w_2)$
and  $ y_2=w_2+z_2-x_2$.  The result follows by induction.
\end{proof}

\begin{corollary}
In Definition \ref{def2}, $X\cdot Y$ is a Young tableau.
\end{corollary}

\subsection{ Expressions of reductions}\label{mulexp}

In order to prove the Theorem \ref{mainTH}, we have to check that
all possible compositions in $\Gamma$, which are only intersections,
are trivial.

Let $R=(r_1,r_2,\dots,r_n),\ S=(s_1,s_2,\dots,s_n),\
T=(t_1,t_2,\dots,t_n)\in U$.

For reductions, we will use the following notation.
\begin{eqnarray*}
RST&=&\left( \begin{array}{ccccc}
               r_1 & r_2 & r_3 & \dots & r_n \\
               s_1 & s_2 & s_3 & \dots & s_n \\
               t_1 & t_2 & t_3 & \dots & t_n
               \end{array} \right)
        \overset{\text{$S\cdot T=A\cdot B$}}{\Longrightarrow}\left( \begin{array}{ccccc}
                r_1 & r_2 & r_3 & \dots & r_n \\
                a_1 & a_2 & a_3 & \dots & a_n \\
                b_1 & b_2 & b_3 & \dots & b_n
                \end{array} \right)\\
    &\ \overset{\text{$R\cdot A=C\cdot D$}}{\Longrightarrow}&\left( \begin{array}{ccccc}
        c_1 & c_2 & c_3 & \dots & c_n \\
        d_1 & d_2 & d_3 & \dots & d_n \\
        b_1 & b_2 & b_3 & \dots & b_n
        \end{array} \right)
    \overset{\text{$D\cdot B=E\cdot F$}}{\Longrightarrow}\left( \begin{array}{ccccc}
    c_1 & c_2  & c_3 & \dots & c_n \\
    e_1 & e_2 & e_3 & \dots & e_n \\
    f_1 & f_2 & f_3 & \dots & f_n
    \end{array} \right),
\end{eqnarray*}
\begin{eqnarray*}
RST&=&\left( \begin{array}{ccccc}
               r_1 & r_2 & r_3 & \dots & r_n \\
               s_1 & s_2 & s_3 & \dots & s_n \\
               t_1 & t_2 & t_3 & \dots & t_n
               \end{array} \right)
         \overset{\text{$R\cdot S=G\cdot H$}}{\Longrightarrow}\left( \begin{array}{ccccc}
                g_1 & g_2 & g_3 & \dots & g_n \\
                h_1 & h_2 & h_3 & \dots & h_n \\
                t_1 & t_2 & t_3 & \dots & t_n
                \end{array} \right)\\
    &\overset{\text{$H\cdot T=I\cdot J$}}{\Longrightarrow}&\left( \begin{array}{ccccc}
        g_1 & g_2 & g_3 & \dots & g_n \\
        i_1 & i_2 & i_3 & \dots & i_n \\
        j_1 & j_2 & j_3 & \dots & j_n
        \end{array} \right)
    \overset{\text{$G\cdot I=K\cdot L$}}{\Longrightarrow}\left( \begin{array}{ccccc}
    k_1 & k_2 & k_3 & \dots & k_n \\
    l_1 & l_2 & l_3 & \dots & l_n \\
    j_1 & j_2 & j_3 & \dots & j_n
    \end{array} \right),
\end{eqnarray*}
where
\begin{eqnarray*}
&&A=(a_1,a_2,\dots,a_n),\ \ B=(b_1,b_2,\dots,b_n),\ \
C=(c_1,c_2,\dots,c_n),\\
&&D=(d_1,d_2,\dots,d_n),\ \ E=(e_1,e_2,\dots,e_n),\ \
F=(f_1,f_2,\dots,f_n),\\
&&G=(g_1,g_2,\dots,g_n),\ \ H=(h_1,h_2,\dots,h_n),\ \
I=(i_1,i_2,\dots,i_n),\\
&&J=(j_1,j_2,\dots,j_n),\ \ K=(k_1,k_2,\dots,k_n),\ \
L=(l_1,l_2,\dots,l_n),
\end{eqnarray*}
and $S\cdot T=A\cdot B,R\cdot A=C\cdot D,D\cdot B=E\cdot F, R\cdot
S=G\cdot H,H\cdot T=I\cdot J,G\cdot I=K\cdot L\in \Gamma$.

We will prove that
\begin{equation*}\label{e1}
\left( \begin{array}{ccccc}
c_1 & c_2  & c_3 & \dots & c_n \\
    e_1 & e_2 & e_3 & \dots & e_n \\
    f_1 & f_2 & f_3 & \dots & f_n
    \end{array} \right)
=\left( \begin{array}{ccccc}
    k_1 & k_2 & k_3 & \dots & k_n \\
    l_1 & l_2 & l_3 & \dots & l_n \\
    j_1 & j_2 & j_3 & \dots & j_n
    \end{array} \right),
\end{equation*}
which implies that the intersection composition $(RS,ST)_{RST}\equiv
0, mod(\Gamma,RST)$. Therefore,
 $\Gamma$ is a
Gr\"{o}bner-Shirshov basis of the algebra $F\langle U\mid
\Gamma\rangle$.

By the definition of the algorithm in Definition \ref{def2},
$a_1=c_1=c_2=e_1=g_1=i_1=k_1=k_2=l_1=0$.
Therefore, one needs only to show that $c_p=k_p$ and $e_q=l_q$ for all $3\le p\le
n,\ 2\le q\le n$.

\subsection{$c_p=k_p\ (p\ge 3)$} \label{cpkp}

We need the following lemmas to prove $c_p=k_p\ (p\ge 3)$.

\begin{lemma}\label{AleI}
For all $p\ (1\leq p\le n)$, $A_p\le I_p$.
\end{lemma}
\begin{proof} By the algorithm in Definition \ref{def2},
$A_1=a_1=i_1=I_1=0$. Assume that for all $p\le m-1\ (m\ge 2)$,
$A_p\le I_p$. Since $A_m=min(T_{m-1},A_{m-1}+s_m),\
I_m=min(T_{m-1},I_{m-1}+h_m)$ and $h_m=r_m+s_m-g_m
   =r_m+s_m-min(S_{m-1}-G_{m-1},r_m)\ge s_m$,
we have $A_m\le I_m$. Now by induction, the result follows.
\end{proof}

\begin{lemma}\label{KeqI}
Assume that $p\geq3, \ K_{p-1}=C_{p-1},\
K_p=C_p=A_{p-1}<C_{p-1}+r_p$ and $S_{p-1}-G_{p-1}\ge r_p$. Then
$K_p=I_{p-1}$.
\end{lemma}
\begin{proof} Note that $ C_p=min(A_{p-1},C_{p-1}+r_p)$ and $
K_p=min(I_{p-1},K_{p-1}+g_p)$. Since $S_{p-1}-G_{p-1}\ge r_p,$
$g_p=min(S_{p-1}-G_{p-1},r_p)=r_p.$ Therefore
$K_p=C_p=A_{p-1}<C_{p-1}+r_p=K_{p-1}+g_p$, which concludes
$K_p=I_{p-1}$.
\end{proof}

\begin{lemma}\label{CCr2SGr}
Assume that  $p\geq2,\ C_p=K_p,\ C_{p+1}=K_{p+1}$ and
$C_{p+1}=C_p+r_{p+1}$. Then $S_p-G_p\ge r_{p+1}$.
\end{lemma}
\begin{proof} Note that $C_{p+1}=min(A_p,C_p+r_{p+1})$ and
$K_{p+1}=min(I_p,K_p+g_{p+1})
      =min(I_p,K_p+S_p-G_p,K_p+r_{p+1})$.
Since $C_p=K_p,  C_{p+1}=K_{p+1}$ and $C_{p+1}=C_p+r_{p+1}$, $
K_{p+1}=C_{p+1}
        =C_p+r_{p+1}
        =K_p+r_{p+1}.
$

Therefore, $I_p\ge K_p+r_{p+1}$, $K_p+S_p-G_p\ge K_p+r_{p+1}$, which
concludes $S_p-G_p\ge r_{p+1}$.
\end{proof}

\subsubsection{$C_p=K_p\ (3\leq p\leq n)$} \label{c3k3}

\begin{proof} Induction on $p$.

Since $c_1=c_2=k_1=k_2=0$, we have  $C_3=c_3$ and $K_3=k_3$.
According to the algorithm in Definition \ref{def2}, $
c_3=min(a_2,r_3)
   =min[min(t_1,s_2),r_3]
   =min(t_1,s_2,r_3)$ and $k_3=min(i_2,g_3)
   =min[min(t_1,h_2),min(s_1+s_2-g_2,r_3)]
   =min[t_1,r_2+s_2-min(s_1,r_2),s_1+s_2-min(s_1,r_2),r_3]
   =min[t_1,s_2+min(r_2,s_1)-min(s_1,r_2),r_3]
   =min(t_1,s_2,r_3).
$
Therefore, $C_3=K_3=min(t_1,s_2,r_3)$.

Assume that for all $3\leq p\le m-1$, $C_p=K_p$. Note that
\begin{eqnarray*}
&&C_m=min(T_{m-2},A_{m-2}+s_{m-1},C_{m-1}+r_m), \\
&&K_m=min(T_{m-2},I_{m-2}+h_{m-1},K_{m-1}+M,K_{m-1}+r_m),
\end{eqnarray*}
where $M=s_{m-1}+S_{m-2}-G_{m-2}-g_{m-1}$.

Observe the following equations
\begin{eqnarray*}
C_{m-t}=min(A_{m-t-1},C_{m-t-1}+r_{m-t}), \ \ \ 1\leq t\leq m-3
\end{eqnarray*}
and denote $q$ the cardinality of the following set
$$
\{i|C_{m-j}=C_{m-j-1}+r_{m-j},\ 1\leq j\leq i\leq m-3\}.
$$

There are three cases to consider.

Case 1. $q=0$, i.e.  $C_{m-1}=A_{m-2}<C_{m-2}+r_{m-1}$.

If $S_{m-2}-G_{m-2}<r_{m-1}$, then
$g_{m-1}=S_{m-2}-G_{m-2}, \
h_{m-1}=r_{m-1}+s_{m-1}-S_{m-2}+G_{m-2}>s_{m-1},\ M=s_{m-1}$ and $
I_{m-2}+h_{m-1}\ge K_{m-1}+h_{m-1}>K_{m-1}+s_{m-1}. $

Therefore,
$K_m=C_m=min(T_{m-2},C_{m-1}+s_{m-1},C_{m-1}+r_m)$.

If $S_{m-2}-G_{m-2}\ge r_{m-1}$, then
$g_{m-1}=r_{m-1},\ h_{m-1}=s_{m-1}$ and $M\ge
s_{m-1}$. By Lemma \ref{KeqI}, we have $ K_{m-1}=I_{m-2}$ and $
I_{m-2}+h_{m-1}=K_{m-1}+s_{m-1}$.

Therefore,
$K_m=C_m=min(T_{m-2},C_{m-1}+s_{m-1},C_{m-1}+r_m)$.

Case 2. $1\leq q\leq m-4$, i.e. for all $1\leq k\le q$, $C_{m-k}=C_{m-k-1}+r_{m-k}$, $C_{m-k-1}+r_{m-k}\le A_{m-k-1}$, $A_{m-q-2}< C_{m-q-2}+r_{m-q-1}$ and $C_{m-1}=A_{m-q-2}+\sum_{i=1}^{q}r_{m-i}$.
Then $ C_{m-k}=C_{m-k-1}+r_{m-k} $ and by Lemma
\ref{CCr2SGr}, we have $S_{m-k-1}-G_{m-k-1}\ge r_{m-k}, \
g_{m-k}=r_{m-k}, \ h_{m-k}=s_{m-k}$,
\begin{eqnarray*}
&&C_m=min(T_{m-2},T_{m-3}+s_{m-1},\dots,T_{m-q-2}+\sum_{i=1}^{q}s_{m-i},
A_{m-q-2}+\sum_{i=1}^{q+1}s_{m-i},C_{m-1}+r_m),\\
&&K_m=min(T_{m-2},T_{m-3}+s_{m-1},\dots,T_{m-q-2}+\sum_{i=1}^{q}s_{m-i},\\
&&\ \ \ \ \ \ \ \ \ \ \ \ \ \ \ \ \
I_{m-q-2}+h_{m-q-1}+\sum_{i=1}^{q}s_{m-i},
K_{m-q-1}+\sum_{i=1}^{q}r_{m-i}+M,K_{m-1}+r_m),
\end{eqnarray*}
where
$M=\sum_{i=1}^{q+1}s_{m-i}+S_{m-q-2}-G_{m-q-2}-g_{m-q-1}-\sum_{i=1}^{q}r_{m-i}$.

If $A_{m-q-2}<C_{m-q-2}+r_{m-q-1}$ and
$S_{m-q-2}-G_{m-q-2}<r_{m-q-1}$, then
\begin{eqnarray*}
&&C_{m-q-1}=A_{m-q-2},\\
&&g_{m-q-1}=S_{m-q-2}-G_{m-q-2},\\
&&h_{m-q-1}=r_{m-q-1}+s_{m-q-1}-S_{m-q-2}+G_{m-q-2}>s_{m-q-1},\\
&&M=\sum_{i=1}^{q+1}s_{m-i}-\sum_{i=1}^{q}r_{m-i},\\
&&K_{m-q-1}+\sum_{i=1}^{q}r_{m-i}+M=K_{m-q-1}+\sum_{i=1}^{q+1}s_{m-i},\\
&&I_{m-q-2}+h_{m-q-1}+\sum_{i=1}^{q}s_{m-i}>K_{m-q-1}+\sum_{i=1}^{q+1}s_{m-i}.
\end{eqnarray*}
Therefore,
\begin{eqnarray*}
K_m&=&min(T_{m-2},T_{m-3}+s_{m-1},\dots,T_{m-q-2}+\sum_{i=1}^{q}s_{m-i},
C_{m-q-1}+\sum_{i=1}^{q+1}s_{m-i},C_{m-1}+r_m)\\
   &=&C_m.
\end{eqnarray*}

If $A_{m-q-2}<C_{m-q-2}+r_{m-q-1}$ and
$S_{m-q-2}-G_{m-q-2}\ge r_{m-q-1}$, then $C_{m-q-1}=A_{m-q-2},$
$g_{m-q-1}=r_{m-q-1}$ and $h_{m-q-1}=s_{m-q-1}$. By Lemma
\ref{KeqI},
\begin{eqnarray*}
&&K_{m-q-1}=I_{m-q-2},\\
&&M\ge \sum_{i=1}^{q+1}s_{m-i}-\sum_{i=1}^{q}r_{m-i},\\
&&K_{m-q-1}+\sum_{i=1}^{q}r_{m-i}+M\ge
K_{m-q-1}+\sum_{i=1}^{q+1}s_{m-i},\\
&&I_{m-q-2}+h_{m-q-1}+\sum_{i=1}^{q}s_{m-i}=K_{m-q-1}+\sum_{i=1}^{q+1}s_{m-i}.
\end{eqnarray*}
Therefore,
\begin{eqnarray*}
K_m=min(T_{m-2},T_{m-3}+s_{m-1},\dots,
T_{m-q-2}+\sum_{i=1}^{q}s_{m-i},C_{m-q-1}+\sum_{i=1}^{q+1}s_{m-i},C_{m-1}+r_m)=C_m.
\end{eqnarray*}

Case 3. $q=m-3$, i.e. for all $1\leq k\le q=m-3$, $C_{m-k}=C_{m-k-1}+r_{m-k}$,
$C_{m-k-1}+r_{m-k}\le A_{m-k-1}$ which implies $C_{m-1}=C_2+\sum_{i=1}^{m-3}r_{m-i}$.
Now, by Lemma \ref{CCr2SGr},
$S_{m-k-1}-G_{m-k-1}\ge r_{m-k},\ g_{m-k}=r_{m-k},\ h_{m-k}=s_{m-k}$
and $ M=\sum_{i=1}^{m-2}s_{m-i}+S_1-G_1-g_2-\sum_{i=1}^{m-3}r_{m-i}
 =\sum_{i=1}^{m-2}s_{m-i}+s_1-g_2-\sum_{i=1}^{m-3}r_{m-i}$.
Thus, we have
\begin{eqnarray*}
C_m&=&min(T_{m-2},T_{m-3}+s_{m-1},\dots,T_1+\sum_{i=1}^{m-3}s_{m-i},A_1+\sum_{i=1}^{m-2}s_{m-i},C_{m-1}+r_m)\\
   &=&min(T_{m-2},T_{m-3}+s_{m-1},\dots,T_1+\sum_{i=1}^{m-3}s_{m-i},\sum_{i=1}^{m-2}s_{m-i},C_{m-1}+r_m)\
   \mbox{ and }\\
K_m&=&min(T_{m-2},T_{m-3}+s_{m-1},\dots,T_1+\sum_{i=1}^{m-3}s_{m-i},I_1+h_2+\sum_{i=1}^{m-3}s_{m-i},\\
   & &\ \ \ \ \ \ \ K_2+\sum_{i=1}^{m-3}r_{m-i}+M,K_{m-1}+r_m)\\
   &=&min(T_{m-2},T_{m-3}+s_{m-1},\dots,T_1+\sum_{i=1}^{m-3}s_{m-i},\\
        & & \ \ \ \ \ \ \
        \sum_{i=1}^{m-2}s_{m-i}+min(h_2-s_2,s_1-g_2),K_{m-1}+r_m).
\end{eqnarray*}
Since
$min(h_2-s_2,s_1-g_2)=min(r_2-min(s_1,r_2),s_1-min(s_1,r_2))=0$, one
gets
$
K_m=min(T_{m-2},T_{m-3}+s_{m-1},\dots,T_1+\sum_{i=1}^{m-3}s_{m-i},\sum_{i=1}^{m-2}s_{m-i},K_{m-1}+r_m)=C_m.
$

So, the proof of $K_m=C_m$ is complete.
\end{proof}

\ \

Therefore, we have shown that $C_p=K_p\ (1\leq p\leq n)$ which, by
using the formulas in Definition \ref{def2}, is clearly equivalent
to $c_p=k_p\ (1\leq p\leq n)$.

\subsection{$e_p=l_p$ $(p\ge 2)$}

We need the following lemmas to prove $e_p=l_p\ (p\ge 2)$.

\begin{lemma}\label{SCGAgezero}
For any $p\ (p\le n)$, $S_p+C_p-G_p-A_p\ge 0$.
\end{lemma}
\begin{proof} For $p=1$, we have
$S_{1}+C_{1}-G_{1}-A_{1}=s_1\ge 0$. Assume that for all $p\le m-1\
(m\ge 2)$, $S_p+C_p-G_p-A_p\ge 0$.

If $A_{m-1}<C_{m-1}+r_m$, then $C_m=A_{m-1}$ and
$
S_m+C_m-G_m-A_m \ge S_m+A_{m-1}-S_{m-1}-A_{m-1}-s_m=0.
$

If $C_{m-1}+r_m\le A_{m-1}$, then $C_m=C_{m-1}+r_m$. By Lemma
\ref{CCr2SGr}, we have $S_{m-1}-G_{m-1}\ge r_m,\ \ \
G_m=G_{m-1}+r_m.$ So,
$S_m+C_m-G_m-A_m=S_m+C_{m-1}+r_m-G_{m-1}-r_m-A_m
               =S_m+C_{m-1}-G_{m-1}-A_m
               =S_{m-1}+C_{m-1}-G_{m-1}-A_{m-1}+s_m-a_m
               \ge  s_m-a_m
               =s_m-min(T_{m-1}-A_{m-1},s_m)
               \ge  0.$

The proof is complete.
\end{proof}

\begin{lemma}\label{AeqI}
If $S_p+C_p-G_p-A_p>0$, then  $A_p=I_p$.
\end{lemma}

\begin{proof}
For $p=1$, we have $A_1=I_1=0$. Assume that for all $p\le m-1\ (m\ge
2)$, $A_p=I_p$ if $S_p+C_p-G_p-A_p>0$. By using Lemma
\ref{SCGAgezero}, $S_m+C_m-G_m-A_m\ge 0$. Suppose that
$S_m+C_m-G_m-A_m>0$. Then by the proof of Lemma \ref{SCGAgezero},
one of the following four conditions should be satisfied.

(i) $A_{m-1}<C_{m-1}+r_m$, $A_m=T_{m-1}<A_{m-1}+s_m$;

(ii) $A_{m-1}<C_{m-1}+r_m$, $G_m=G_{m-1}+r_m<S_{m-1}$;

(iii) $C_{m-1}+r_m\le A_{m-1}$, $S_{m-1}+C_{m-1}-G_{m-1}-A_{m-1}>0$;

(iv) $C_{m-1}+r_m\le A_{m-1}$, $S_{m-1}+C_{m-1}-G_{m-1}-A_{m-1}=0$,
$s_m>a_m$.

Assume (i). Then $ C_m=A_{m-1}$, and by Lemma \ref{AleI},
                       $A_m=T_{m-1}<A_{m-1}+s_m\le I_{m-1}+h_m$. So,
                       $ A_m=I_m=T_{m-1}$.

Assume (ii). Then $C_m=A_{m-1},\ g_m=r_m,\ h_m=s_m$, and by
                      Lemma \ref{KeqI}, $K_m=I_{m-1}$. Thus,
                      $A_m=I_m=min(T_{m-1},C_m+s_m)$.

Assume (iii). Then $C_m=C_{m-1}+r_m$. Since $K_m=C_m$ and
$K_{m-1}=C_{m-1}$, by Lemma \ref{CCr2SGr}, we have
                       $S_{m-1}-G_{m-1}\ge r_m$, $g_m=r_m$, $h_m=s_m$ and
                       $A_{m-1}=I_{m-1}$. So,
                       $A_m=I_m=min(T_{m-1},A_{m-1}+s_m)$.

Assume (iv). Then $C_m=C_{m-1}+r_m$ and $a_m=T_{m-1}-A_{m-1}$. By
Lemma \ref{AleI}, $A_m=T_{m-1}<A_{m-1}+s_m\le I_{m-1}+h_m$. So,
                       $A_m=I_m=T_{m-1}$.
\end{proof}

\ \

We use induction on $p$ to prove that $e_p=l_p$ $(2\leq p\leq n)$.

\subsubsection{$e_2=l_2$} \label{e2l2}

\begin{proof}
$e_2=min(b_1,d_2)=min(s_1+t_1,r_2+a_2)=min(s_1+t_1,r_2+min(t_1,s_2))=min(s_1+t_1,t_1+r_2,r_2+s_2)=
min(s_1,r_2)+min(t_1,r_2+s_2-min(s_1,r_2))=min(s_1,r_2)+min(t_1,h_2)=g_2+i_2=l_2$.
\end{proof}

\subsubsection{$e_p=l_p$ $(p\ge 3)$} \label{enln}

\begin{proof}
Assume that for any $p\le m-1\ (m\ge 3)$, $e_p=l_p$.
Since $E_{m-1}=L_{m-1}=G_{m-1}+I_{m-1}-K_{m-1}$ and
$K_{m-1}=C_{m-1}=C_m-c_m$, we have
\begin{eqnarray*}
e_m&=&min(B_{m-1}-E_{m-1},d_m)\\
    &=&min(S_{m-1}+T_{m-1}-A_{m-1}-E_{m-1},r_m+a_m-c_m)\\
    &=&min(S_{m-1}+T_{m-1}-A_{m-1}-E_{m-1},r_m+min(T_{m-1}-A_{m-1},s_m)-c_m)\\
    &=&min(S_{m-1}+T_{m-1}+C_m-A_{m-1}-G_{m-1}-I_{m-1},r_m+T_{m-1}-A_{m-1},\\
    & &\ \ \ \ r_m+s_m)-c_m,\\
l_m&=&g_m+i_m-k_m\\
    &=&min(S_{m-1}-G_{m-1},r_m)+min(T_{m-1}-I_{m-1},h_m)-k_m\\
    &=&min(S_{m-1}+T_{m-1}-G_{m-1}-I_{m-1},S_{m-1}-G_{m-1}+h_m,\\
    & &\ \ \ \ r_m+T_{m-1}-I_{m-1},r_m+h_m)-k_m.
\end{eqnarray*}
There are two cases to consider.

Case 1. $A_{m-1}<C_{m-1}+r_m$.
Then $C_m=A_{m-1}$.

If $S_{m-1}-G_{m-1}<r_m$, then $g_m=S_{m-1}-G_{m-1},\
h_m=r_m+s_m-S_{m-1}+G_{m-1}>s_m,\
S_{m-1}-G_{m-1}+h_m=r_m+s_m<r_m+h_m
$
and by Lemma \ref{AleI},
\begin{eqnarray*}
S_{m-1}+T_{m-1}-G_{m-1}-I_{m-1}&<&r_m+T_{m-1}-I_{m-1}
                               \le r_m+T_{m-1}-A_{m-1}.
\end{eqnarray*}
Therefore,
$l_m=e_m=min(S_{m-1}+T_{m-1}-G_{m-1}-I_{m-1},r_m+s_m)-c_m$.

If $S_{m-1}-G_{m-1}\ge r_m$, then $g_m=r_m,\ h_m=s_m$ and by Lemma
\ref{KeqI}, we have $K_m=I_{m-1},\
S_{m-1}+T_{m-1}-G_{m-1}-I_{m-1}\ge
r_m+T_{m-1}-I_{m-1}=r_m+T_{m-1}-K_m$ and
$S_{m-1}-G_{m-1}+h_m=S_{m-1}-G_{m-1}+s_m\ge r_m+s_m$. Therefore,
$l_m=e_m=min(T_{m-1}-C_m,s_m)+r_m-c_m$.

Case 2. $C_{m-1}+r_m\le A_{m-1}$.
Then $C_m=C_{m-1}+r_m$, and by Lemma \ref{CCr2SGr}, we have
$S_{m-1}-G_{m-1}\ge r_m,\ g_m=r_m,\ h_m=s_m,\
S_{m-1}+T_{m-1}-G_{m-1}-I_{m-1}\ge r_m+T_{m-1}-I_{m-1},\
S_{m-1}-G_{m-1}+h_m\ge r_m+h_m=r_m+s_m,\
e_m=min(S_{m-1}+T_{m-1}+C_{m-1}-A_{m-1}-G_{m-1}-I_{m-1},T_{m-1}-A_{m-1},s_m)+r_m-c_m$
and $l_m=min(T_{m-1}-I_{m-1},s_m)+r_m-k_m$.

By Lemma \ref{SCGAgezero}, $S_{m-1}+C_{m-1}-G_{m-1}-A_{m-1}\ge 0$.

If $S_{m-1}+C_{m-1}-G_{m-1}-A_{m-1}=0$, by Lemma \ref{AleI}, $
e_m=min(T_{m-1}-I_{m-1},T_{m-1}-A_{m-1},s_m)+r_m-c_m
   =min(T_{m-1}-I_{m-1}, s_m)+r_m-k_m
  =l_m$.

If $S_{m-1}+C_{m-1}-G_{m-1}-A_{m-1}>0$, by Lemma \ref{AeqI}, we have
$A_{m-1}=I_{m-1}$ and
$S_{m-1}+T_{m-1}+C_{m-1}-A_{m-1}-G_{m-1}-I_{m-1}>T_{m-1}-A_{m-1}$.
Therefore, $e_m=min(T_{m-1}-A_{m-1},s_m)+r_m-c_m
=min(T_{m-1}-I_{m-1},s_m)+r_m-k_m =l_m$.
\end{proof}

\section{Plactic monoid with column generators}\noindent

Let $A=\{1,2,\dots,n\}$. Recall that a strictly decreasing word $C\in A^*$ is
called a column, for example, 6531 is a column.

Let $R,S \in A^*$ be two columns, $R=i_ti_{t-1}\dots i_2i_1$,
$S=j_lj_{l-1}\dots j_2j_1$. We define $R \rhd S$ if $t\geq l$  and
$i_k\leq j_k,\ k=1,2,\dots,l$.

A (semistandard) Young tableau on A is a word $w=R_1R_2\dots R_t$
such that $R_i\rhd R_{i+1},\ i=1,2,\dots t-1$, where each $R_i$ is a
column. For example,
$$
421\cdot 521\cdot531\cdot632\cdot54\cdot74\cdot4
$$
is a Young tableau.

{\bf Remark:} We use two ways to define a Young tableau which are
essentially the same.

Let $C\in A^*$ be a column and $c_i$  the number of letter $i\ (1\le
i\le n)$ in $C$. Then $c_i\in \{0,1\},\ i=1,2,\dots,n$. We denote
$C=(c_1;c_2;\dots;c_n)$. For example,
$C=6421=(1;1;0;1;0;1;0;\dots;0).$

Let $V=\{C\in A^*\ |\ C \mbox{ is a column}\}$.

Let $R=(r_1;r_2;\dots;r_n )\in V$ and $wt(R)=(|R|,r_1,\dots,r_n)$.
We order $V$: for any $R, S\in V$, $ R<S $ if and  only if
$wt(R)>wt(S)$ lexicograghically. Then, we order $V^*$ by the deg-lex
ordering. We will use this ordering throughout this section.

Denote
$$
R_0=0,\ \ \ R_p=\sum_{i=1}^{p} r_i,\ 1\le p \le n.
$$
\begin{lemma}\label{lemma9}
For any $R,S\in V$, $R\rhd S$  if and only if $R_p\ge S_p,\ 1\le p \le n$.
\end{lemma}
\begin{proof}
Let $R=i_ti_{t-1}\dots
i_2i_1,\ S=j_lj_{l-1}\dots j_2j_1 \in V$. Then $t=R_n,l=S_n$.

Suppose that $R\rhd S$. Then by definition, we have $t\ge l$ and $i_1\le
j_1,i_2\le j_2,\dots,i_l\le j_l$. Note that $R=(r_1;r_2;\dots;r_n
)$, $r_p=1 $ if $p=i_a$ $(a=1,\dots,t)$; $S=(s_1;s_2;\dots;s_n ),\
s_p=1 $ if $ p=j_b $ $(b=1,\dots,l)$.

For $1\le k \le i_1-1$, $R_k=S_k=0$, so $R_k\ge S_k$; for $i_1\le k
\le i_2-1\ (\le j_2-1)$, $R_k=1,S_k\le 1$, so $R_k\ge S_k$; \dots;
for $i_{l-1}\le k \le i_l-1 \ (\le j_l-1)$, $R_k=l-1,S_k\le l-1$, so
$R_k\ge S_k$; for $i_l\le k \le n,\ R_k\ge l,S_k\le l$, so $R_k\ge
S_k$. Therefore, $R_p\ge S_p ,1\le p \le n$.

Suppose that $R_p\ge S_p,\ 1\le p \le n$. Then $t\ge l$. Since
$R_{i_1}=1$, $R_{j_1}\ge S_{j_1}=1$, we have $i_1\le j_1$.
Similarly, $i_2\le j_2,\dots,i_l\le j_l$. Therefore $R\rhd S$.
\end{proof}

\begin{corollary}\label{youngtableau}
For any $R,S,T\in V,\ w=RST$ is a Young tableau if and only if
$R_p\ge S_p\ge T_p,\ 1\le p \le n$.
\end{corollary}

\begin{definition}\label{defc1}
(Robinson-Schensted column algorithm)  Let $R\in V$, $x\in$ A.
$$
x\cdot R = \left\{
\begin{array}{cc}
    xR,& \mbox{ if }\ xR \mbox{ is a column}; \\
    R'\cdot y,& \mbox{otherwise.}
\end{array}
\right.
$$
where $y$ is the rightmost letter in $R$ and is  larger than or
equal to $x$, and $R'=R\mid _{y\rightarrow x}$, i.e. $R'$ is
obtained from $R$ by replacing $y$ by $x$.
\end{definition}

\begin{lemma}\label{exist}
For any $R,S\in V$, by  Robinson-Schensted column algorithm, there exist
$R',S'\in V$ such that $R\cdot S=R'\cdot S'$ and $R'\cdot S'$ is a
Young tableau, where $S'$ is empty if $R\cdot S=R'$ is a column.
\end{lemma}
\begin{proof} Suppose $R=i_ti_{t-1}\dots i_2i_1,\ S=j_lj_{l-1}\dots j_2j_1$.

If $i_1>j_l$, we have $RS$ is a column which is a one-column Young
tableau.

If $i_1\le j_l$, then there uniquely exists a $y_1$ and $y_1$ is the
rightmost letter in $S$, such that $y_1\ge i_1$. By  Robinson-Schensted
column algorithm, we have $i_1\cdot S=S\mid _{y_1\rightarrow
i_1}\cdot y_1$, and $S\mid _{y_1\rightarrow i_1}\cdot y_1$ has two
columns. Since the minimal number in $S\mid _{y_1\rightarrow i_1}$
is $\le i_1$ and $i_1\le y_1$, $S\mid _{y_1\rightarrow i_1}\cdot
y_1$ is also a two-column Young tableau.

For $i_2\cdot S\mid _{y_1\rightarrow i_1}\cdot y_1$, there are two
cases to consider.

Case 1. $i_2S\mid _{y_1\rightarrow i_1}$ is a column. Then $i_2\cdot
S\mid _{y_1\rightarrow i_1}\cdot y_1=i_2S\mid _{y_1\rightarrow
i_1}\cdot y_1$ and $i_2S\mid _{y_1\rightarrow i_1}\cdot y_1$ is a
two-column Young tableau.

Case 2. There exists a $y_2$ in $S$ such that  $y_2$ is the
rightmost letter in $S\mid _{y_1\rightarrow i_1}$ with $y_2\ge i_2$.
Since $y_2\ge i_2>i_1$ and $y_1$ is the rightmost letter in $S$ that
is $\ge i_1$, we have $y_2>y_1$. Then $i_2\cdot S\mid
_{y_1\rightarrow i_1}\cdot y_1=S\mid _{y_2\rightarrow
i_2,y_1\rightarrow i_1}\cdot y_2y_1$. Since the second minimal
number in $S\mid _{y_2\rightarrow i_2,y_1\rightarrow i_1}$ (also in
$S$) is $\le i_2\ (\le y_2)$, $S\mid _{y_2\rightarrow
i_2,y_1\rightarrow i_1}\cdot y_2y_1$ is also a two-column Young
tableau.

Continuing in this way, we will find two columns $R',S'\in V$ such
that $R\cdot S=R'\cdot S'$ and $R'\cdot S'$ is a Young tableau,
where $S'$ is empty if $R\cdot S=R'$ is a column.
\end{proof}

Denote
$$
\Lambda=\{R\cdot S=R'\cdot S',\ R,S\in V\}.
$$
By noting that in $sgp\langle A\mid \Omega\rangle,\ R\cdot S=R'\cdot
S'$,  it follows that $ sgp\langle V\mid \Lambda\rangle\cong
sgp\langle A\mid \Omega\rangle $ and so we may assume that $F\langle
V\mid \Lambda\rangle=F\langle A\mid \Omega\rangle$.

It is clear that $\Lambda$ is a finite set.

The following Theorem \ref{maintheorem2} is the second main result
of the paper.

\begin{theorem}\label{maintheorem2}
With the deg-lex ordering on $V^*$, $\Lambda$ is a finite
Gr\"{o}bner-Shirshov basis for the plactic algebra $F\langle V\mid
\Lambda\rangle$. The set of Young tableaux on $A$ is a normal form
of the plactic monoid $sgp\langle A\mid \Omega\rangle$.
\end{theorem}

{\bf Remark}: For an arbitrary well-ordered set $A$, we
define similarly  a column on $A$, Robinson-Schensted column algorithm
and the set  $\Lambda$. Then for an arbitrary well-ordered set $A$,
similar to the proof of $A$ to be finite, we also have Theorem
\ref{maintheorem2}. If this is the case,  $\Lambda$ may not be
finite.

\subsection{Main formula for the product of column generators}

\begin{definition} \label{defc2}
Let $W=(w_1;w_2;\dots;w_n),\ Z=(z_1;z_2;\dots;z_n)\in V$. Define an
algorithm $W\cdot Z=W'\cdot Z'$, where $W'=(w_1';w_2';\dots;w_n'),\
Z'=(z_1';z_2';\dots;z_n')$, $z_p'=min(W_p-Z'_{p-1},z_p),\
w_p'=w_p+z_p-z'_p\ ( n\geq p\ge 1)$. Then $W',Z'\in V$ or $W'\in V$
and $Z'$ is empty.
\end{definition}

\begin{lemma}\label{lemma4.5}
The algorithms in Definition \ref{defc1} and Definition \ref{defc2}
are equivalent.
\end{lemma}

The proof of this lemma is similar to that of Lemma \ref{eqv}.

\begin{corollary}
In Definition \ref{defc2}, $W'\cdot Z'$ is a Young tableau.
\end{corollary}

\subsection{ Expressions of reductions}

In order to prove the Theorem \ref{maintheorem2}, we have to check
that all possible compositions in $\Lambda$, which are only
intersections, are trivial.

Let $R=(r_1;r_2;\dots;r_n),\ S=(s_1;s_2;\dots;s_n),\
T=(t_1;t_2;\dots;t_n)\in V$.

For reductions, we will use the following notation.

\begin{eqnarray*}
TSR&=&\left( \begin{array}{ccc}
               t_1 & s_1 & r_1  \\
               t_2 & s_2 & r_2  \\
               \vdots & \vdots & \vdots  \\
               t_n & s_n & r_n
               \end{array} \right)
        \overset{\text{$T\cdot S=B\cdot A$}}{\Longrightarrow}\left( \begin{array}{ccc}
               b_1 & a_1 & r_1  \\
               b_2 & a_2 & r_2  \\
               \vdots & \vdots & \vdots  \\
               b_n & a_n & r_n
               \end{array} \right)\\
    &\ \overset{\text{$A\cdot R=D\cdot C$}}{\Longrightarrow}&\left( \begin{array}{ccc}
               b_1 & d_1 & c_1  \\
               b_2 & d_2 & c_2  \\
               \vdots & \vdots & \vdots  \\
               b_n & d_n & c_n
               \end{array} \right)
    \overset{\text{$B\cdot D=F\cdot E$}}{\Longrightarrow}\left( \begin{array}{ccc}
               f_1 & e_1 & c_1  \\
               f_2 & e_2 & c_2  \\
               \vdots & \vdots & \vdots  \\
               f_n & e_n & c_n
               \end{array} \right),
\end{eqnarray*}

\begin{eqnarray*}
 TSR&=&\left(\begin{array}{ccc}
               t_1 & s_1 & r_1  \\
               t_2 & s_2 & r_2  \\
               \vdots & \vdots & \vdots  \\
               t_n & s_n & r_n
               \end{array} \right)
        \overset{\text{$S\cdot R=H\cdot G$}}{\Longrightarrow}\left( \begin{array}{ccc}
               t_1 & h_1 & g_1  \\
               t_2 & h_2 & g_2  \\
               \vdots & \vdots & \vdots  \\
               t_n & h_n & g_n
               \end{array} \right)\\
    &\ \overset{\text{$T\cdot H=J\cdot I$}}{\Longrightarrow}&\left( \begin{array}{ccc}
               j_1 & i_1 & g_1  \\
               j_2 & i_2 & g_2  \\
               \vdots & \vdots & \vdots  \\
               j_n & i_n & g_n
               \end{array} \right)
    \overset{\text{$I\cdot G=L\cdot K$}}{\Longrightarrow}\left( \begin{array}{ccc}
               j_1 & l_1 & k_1  \\
               j_2 & l_2 & k_2  \\
               \vdots & \vdots & \vdots  \\
               j_n & l_n & k_n
               \end{array} \right),
\end{eqnarray*}
where
\begin{eqnarray*}
&&A=(a_1;a_2;\dots;a_n),\ \ B=(b_1;b_2;\dots;b_n),\ \
C=(c_1;c_2;\dots;c_n),\\
&&D=(d_1;d_2;\dots;d_n),\ \ E=(e_1;e_2;\dots;e_n),\ \
F=(f_1;f_2;\dots;f_n),\\
&&G=(g_1;g_2;\dots;g_n),\ \ H=(h_1;h_2;\dots;h_n),\ \
I=(i_1;i_2;\dots;i_n),\\
&&J=(j_1;j_2;\dots;j_n),\ \ K=(k_1;k_2;\dots;k_n),\ \
L=(l_1;l_2;\dots;l_n),
\end{eqnarray*}
and  $T\cdot S=B\cdot A,A\cdot R=D\cdot C,B\cdot D=F\cdot E,
S\cdot R=H\cdot G, T\cdot H=J\cdot I,I\cdot G=L\cdot K\in \Lambda$.

We will prove that
\begin{equation*}
\left( \begin{array}{ccc}
               f_1 & e_1 & c_1  \\
               f_2 & e_2 & c_2  \\
               \vdots & \vdots & \vdots  \\
               f_n & e_n & c_n
               \end{array} \right)
=\left( \begin{array}{ccc}
               j_1 & l_1 & k_1  \\
               j_2 & l_2 & k_2  \\
               \vdots & \vdots & \vdots  \\
               j_n & l_n & k_n
               \end{array}\right),
\end{equation*}
which implies that the intersection composition
$(TS,SR)_{TSR}\equiv 0, mod(TSR, \Lambda)$. Therefore, $\Lambda$ is
a Gr\"{o}bner-Shirshov basis of the algebra $F\langle V\mid
\Lambda\rangle$.

\subsection{$c_m=k_m\ (1\leq m\leq n)$}

We need the following lemmas to prove $c_m=k_m$.

\begin{lemma}\label{lem1dh}
For all $p\ (1\leq p\le n)$, $I_p\ge A_p$.
\end{lemma}
\begin{proof}
Note that $i_1=min(t_1,h_1)=min(t_1,s_1+r_1-g_1)$ and
$a_1=min(t_1,s_1)$. If $s_1\ge r_1$, then $g_1=r_1$ and $i_1=min(t_1,s_1)=a_1$. If
$s_1<r_1$, then $g_1=s_1$ and $i_1=min(t_1,r_1)\ge min(t_1,s_1)=a_1$. This shows that
$I_1=i_1\ge a_1=A_1$.

Assume that for all $m\le p-1\ (p\ge 2)$, $I_m\ge A_m$. Since
$I_p=I_{p-1}+min(T_p-I_{p-1},h_p)=min(T_p,I_{p-1}+h_p),\ A_p=min(T_p,A_{p-1}+s_p)$ and $
h_p=s_p+r_p-g_p
   =s_p+r_p-min(S_p-G_{p-1},r_p)\ge s_p$,
we have $I_p\ge A_p$. Now by induction, the result follows.
\end{proof}

\begin{lemma}\label{lem2jd}
Assume that $p\geq 2$, $ K_{p-1}=C_{p-1}$, $K_p=C_p=A_p<C_{p-1}+r_p$
and $S_p-G_{p-1}\ge r_p$. Then $C_p=I_p$.
\end{lemma}
\begin{proof}
Note that $ C_p=min(A_p,C_{p-1}+r_p)$ and $
K_p=min(I_p,K_{p-1}+g_p)$. If \ $S_p-G_{p-1}\ge r_p$, then
$g_p=min(S_p-G_{p-1},r_p)=r_p.$

Therefore $K_p=C_p=A_p<K_{p-1}+g_p$, which concludes $C_p=K_p=I_p$.
\end{proof}

\begin{lemma}\label{lem3sb}
Assume that  $p\ge 2$, $K_{p-1}=C_{p-1},\ K_p=C_p$ and
$C_p=C_{p-1}+r_p$. Then $S_p-G_{p-1}\ge r_p$.
\end{lemma}
\begin{proof}
Note that
$K_p=min(I_p,K_{p-1}+g_p)=min(I_p,K_{p-1}+S_p-G_{p-1},K_{p-1}+r_p)$.
Since $ K_{p-1}=C_{p-1}, K_p=C_p$ and $C_p=C_{p-1}+r_p$,
we have
$ K_p=C_p=C_{p-1}+r_p=K_{p-1}+r_p$.

Therefore, $K_{p-1}+S_p-G_{p-1}\ge K_{p-1}+r_p$, which concludes
$S_p-G_{p-1}\ge r_p$.
\end{proof}

\begin{lemma}\label{lem6}
For any $p\ (1\leq p\le n)$, $S_p+C_p-G_p-A_p\ge 0$.
\end{lemma}
\begin{proof} Induction on $p$.

 If $c_1=a_1$, then $s_1+c_1-g_1-a_1=s_1-g_1\ge 0$.
 If $c_1<a_1$, then $c_1=r_1=g_1=0,\ a_1=t_1=s_1=1$. We have
 $s_1+c_1-g_1-a_1=0$. This shows that the result holds for $p=1$.

Assume that for all $1 \leq m\le p-1$, $S_m+C_m-G_m-A_m\ge 0$.

If
$C_p=A_p\le C_{p-1}+r_p$, then $S_p+C_p-G_p-A_p=S_p-G_p\ge 0$.

If
$C_p=C_{p-1}+r_p<A_p$, by Lemma \ref{lem3sb}, $g_p=r_p,h_p=s_p$. So,
\begin{eqnarray*}
S_p+C_p-G_p-A_p&=&S_p+C_{p-1}+r_p-(G_{p-1}+r_p)-A_p\\
&=&S_p+C_{p-1}-G_{p-1}-A_p\\
              &=&S_{p-1}+C_{p-1}-G_{p-1}-A_{p-1}+s_p-a_p\\
             & \ge& s_p-a_p \ge 0.
\end{eqnarray*}
Therefore, $S_p+C_p-G_p-A_p\ge 0$.
\end{proof}

\ \

By noting that $s_p\geq a_p$ and by Lemma \ref{lem6}, we have the following lemma.
\begin{lemma}\label{lem4hjsb}
For all $p\ (1\leq p\le n)$, we have $A_p-C_{p-1}\le S_p-G_{p-1}$.
\end{lemma}

\begin{lemma}\label{lem5as}
If $C_{p-1}<A_{p-1}$, then $h_p=s_p$.

\end{lemma}
\begin{proof}
Note that $h_p=s_p+r_p-g_p=s_p+r_p-min(S_{p-1}-C_{p-1},r_p)\geq s_p$.
If $h_p>s_p$, then $h_p=1$, $s_p=0$,  $r_p=1$ and $S_p-G_{p-1}=0$. By Lemma
\ref{lem4hjsb}, $A_p-C_{p-1}\leq S_p-G_{p-1}=0$. But $A_p-C_{p-1}\geq A_{p-1}-C_{p-1}>0$ which is a contradiction. Therefore,
$h_p=s_p$.
\end{proof}

\subsubsection{$c_p=k_p\ (1\leq p\leq n)$}

It suffices to prove $C_p=K_p\ (1\leq p\leq n)$.
We prove the result by induction on $p$.

Since $k_1=min(i_1,g_1)=min(t_1,h_1,g_1)=min(t_1,s_1,r_1)$ and
$c_1=min(a_1,r_1)=min(t_1,s_1,r_1)$, we have $K_1=C_1$.

Assume that for all $1\leq p\le m-1$, $K_p=C_p$. Note that
\begin{eqnarray*}
C_m&=&min(T_m,A_{m-1}+s_m,C_{m-1}+r_m),\\
K_m&=&min(T_m,I_{m-1}+h_m,K_{m-1}+M,K_{m-1}+r_m),
\end{eqnarray*}
where $M=S_{m}-G_{m-1}=S_{m-1}-G_{m-2}+s_m-g_{m-1}$.

Observe the following equations
\begin{eqnarray*}
&&C_{m-t}=min(A_{m-t},C_{m-t-1}+r_{m-t}),\ \ \ 1\leq t\leq m-2
\end{eqnarray*}
and denote $q$ the cardinality of the following set
$$
\{i|C_{m-j}=C_{m-j-1}+r_{m-j},\ 1\leq j\leq i\leq m-2\}.
$$

There are three cases to  consider.

Case 1. $q=0$, i.e.  $C_{m-1}=A_{m-1}<C_{m-2}+r_{m-1}$.

If $S_{m-1}-G_{m-2}<r_{m-1}$, then
$S_{m-1}-G_{m-2}=g_{m-1}$ and
$K_{m-1}+S_{m-1}-G_{m-2}+s_m-g_{m-1}=K_{m-1}+s_m$. Since $I_{m-1}\ge
K_{m-1}$ and $h_m\ge s_m $, $I_{m-1}+h_m\ge K_{m-1}+s_m$.
Therefore, $K_m=min(T_m,K_{m-1}+s_m,K_{m-1}+r_m)=min(T_m,C_{m-1}+s_m,C_{m-1}+r_m)=min(T_m,A_{m-1}+s_m,C_{m-1}+r_m)=C_m$.

If $S_{m-1}-G_{m-2}\ge r_{m-1}$, then $g_{m-1}=r_{m-1}$. So by Lemma
\ref{lem2jd}, $C_{m-1}=I_{m-1}$,  $S_{m-1}-G_{m-2}+s_m-g_{m-1}\ge
s_m$. If $h_m=s_m$, then $K_m=min(T_m,C_{m-1}+s_m,C_{m-1}+r_m)=C_m$.
If $h_m>s_m$, then $s_m=0,\ h_m=1$. Thus $r_m=1,\ g_m=0$.
This implies $K_m=K_{m-1}=C_{m-1}=C_m$.

Case 2. $1\leq q\leq m-3$, i.e. for any $1\leq k\leq q$, $C_{m-k}=C_{m-k-1}+r_{m-k}$,
$A_{m-k}\ge C_{m-k-1}+r_{m-k}$,
 $C_{m-1}=A_{m-q-1}+\sum_{i=1}^{q}r_{m-i}$ and $A_{m-q-1}< C_{m-q-2}+r_{m-q-1}$.
By Lemma \ref{lem3sb}, we have
$S_{m-k}-G_{m-k-1}\ge r_{m-k}$, $g_{m-k}=r_{m-k}$ and
$h_{m-k}=s_{m-k}$.

We consider the value of $h_m$ first.
Note that $C_{m-2}+r_{m-1}\leq A_{m-1}$.

(i) If $C_{m-2}+r_{m-1}=A_{m-1}$, then $C_{m-1}=A_{m-1}$.

If $S_m-G_{m-1}<r_m$, then $S_m-G_{m-1}=0,\ r_m=1$, $g_m=s_m=0,\
h_m=1$.
Since $I_{m-1}\ge A_{m-1}=K_{m-1}$, we have
\begin{eqnarray*}
C_m&=&min(T_m,C_{m-1},C_{m-1}+r_m)=min(T_m,C_{m-1})\\
K_m&=&min(T_m,I_{m-1}+h_m,K_{m-1},K_{m-1}+r_m)=min(T_m,K_{m-1})
\end{eqnarray*}
Therefore, $C_m=K_m$.

If $S_m-G_{m-1}\ge r_m$, then $g_m=r_m$ and $h_m=s_m$.

(ii) If $C_{m-2}+r_{m-1}<A_{m-1}$, then by Lemma \ref{lem5as}, we have
$h_m=s_m$.

In both cases, we have either $C_m=K_m$ or $h_m=s_m$,
so we may assume that $h_m=s_m$.

Now,
\begin{eqnarray*}
C_m&=&min(T_m,T_{m-1}+s_m,\dots,T_{m-q}+\sum_{i=0}^{q-1}s_{m-i},A_{m-q-1}+\sum_{i=0}^{q}s_{m-i},C_{m-1}+r_m),\\
K_m&=&min(T_m,T_{m-1}+h_m,\dots,T_{m-q}+\sum_{i=0}^{q-1}h_{m-i},I_{m-q-1}+\sum_{i=0}^{q}h_{m-i},\\
&&\ \ \ \ \ \ \ K_{m-q-1}+\sum_{i=1}^{q}r_{m-i}+M,K_{m-1}+r_m)\\
&=&min(T_m,T_{m-1}+s_m,\dots,T_{m-q}+\sum_{i=0}^{q-1}s_{m-i},I_{m-q-1}+\sum_{i=0}^{q}s_{m-i},\\
&&\ \ \ \ \ \ \ K_{m-q-1}+\sum_{i=1}^{q}r_{m-i}+M,K_{m-1}+r_m),
\end{eqnarray*}
where
$M=S_m-G_{m-1}=\sum_{i=0}^{q}s_{m-i}+S_{m-q-1}-G_{m-q-2}-\sum_{i=1}^{q}r_{m-i}-g_{m-q-1}$.

If $S_{m-q-1}-G_{m-q-2}<r_{m-q-1}$, then $C_{m-q-1}=A_{m-q-1}$,
$g_{m-q-1}=S_{m-q-1}-G_{m-q-2}$,
$M=\sum_{i=0}^{q}s_{m-i}-\sum_{i=1}^{q}r_{m-i}$,
$K_{m-q-1}+\sum_{i=1}^{q}r_{m-i}+M=K_{m-q-1}+\sum_{i=0}^{q}s_{m-i}$
and  $I_{m-q-1}+\sum_{i=0}^{q}s_{m-i}\ge
K_{m-q-1}+\sum_{i=0}^{q}s_{m-i}$. Therefore,
\begin{eqnarray*}
K_m&=&min(T_m,T_{m-1}+s_m,\dots,T_{m-q}+\sum_{i=0}^{q-1}s_{m-i},A_{m-q-1}+\sum_{i=0}^{q}s_{m-i},C_{m-1}+r_m)\\
   &=&C_m.
\end{eqnarray*}

If $S_{m-q-1}-G_{m-q-2}\ge r_{m-q-1}$, then $g_{m-q-1}=r_{m-q-1}$.
By Lemma \ref{lem2jd} and $C_{m-q-1}=A_{m-q-1}$, we have $C_{m-q-1}=I_{m-q-1},\ M\ge
\sum_{i=0}^{q}s_{m-i}-\sum_{i=1}^{q}r_{m-i},\
K_{m-q-1}+\sum_{i=1}^{q}r_{m-i}+M\ge
K_{m-q-1}+\sum_{i=0}^{q}s_{m-i}$ and
$A_{m-q-1}+\sum_{i=0}^{q}s_{m-i}=K_{m-q-1}+\sum_{i=0}^{q}s_{m-i}$.
Therefore,
\begin{eqnarray*}
K_m&=&min(T_m,T_{m-1}+s_m,\dots,T_{m-q}+\sum_{i=0}^{q-1}s_{m-i},K_{m-q-1}+\sum_{i=0}^{q}s_{m-i},C_{m-1}+r_m)\\
   &=&C_m .
\end{eqnarray*}

Case 3. $q=m-2$, i.e.  for all $1\leq k\le q=m-2$, $C_{m-k}=C_{m-k-1}+r_{m-k}$,  $C_{m-k-1}+r_{m-k}\le A_{m-k}$ and
$C_{m-1}=C_1+\sum_{i=1}^{m-2}r_{m-i}$.
By Lemma \ref{lem3sb}, $S_{m-k}-G_{m-k-1}\ge r_{m-k}, \
g_{m-k}=r_{m-k}, \ h_{m-k}=s_{m-k}$. So,
\begin{eqnarray*}
C_m&=&min(T_m,T_{m-1}+s_m,\dots,T_1+\sum_{i=0}^{m-2}s_{m-i},\sum_{i=0}^{m-1}s_{m-i},C_{m-1}+r_m)\\
   &=&min(T_m,T_{m-1}+s_m,\dots,T_1+\sum_{i=0}^{m-2}s_{m-i},S_m,C_{m-1}+r_m)
\end{eqnarray*}
and
\begin{eqnarray*}
K_m&=&min(T_m,T_{m-1}+s_m,\dots,T_1+\sum_{i=0}^{m-2}s_{m-i},\sum_{i=0}^{m-2}s_{m-i}+h_1,k_1-g_1+\sum_{i=0}^{m-1}s_{m-i},C_{m-1}+r_m)\\
   &=&min(T_m,T_{m-1}+s_m,\dots,T_1+\sum_{i=0}^{m-2}s_{m-i},S_m+h_1-s_1,S_m+k_1-g_1,C_{m-1}+r_m).
\end{eqnarray*}
If $g_1=k_1$, then $S_m+k_1-g_1=S_m \le S_m+h_1-s_1 $ and
\begin{eqnarray*}
K_m=min(T_m,T_{m-1}+s_m,\dots,T_1+\sum_{i=0}^{m-2}s_{m-i},S_m,C_{m-1}+r_m)=C_m.
\end{eqnarray*}
If $g_1>k_1$, we have $g_1=1,\ k_1=0$, $s_1=r_1=1,\ t_1=0$,
$h_1=s_1=1$ and $T_1+\sum_{i=0}^{m-2}s_{m-i}=S_m+k_1-g_1=S_m-s_1\le
S_m $. So,
\begin{eqnarray*}
K_m=min(T_m,T_{m-1}+s_m,\dots,T_1+\sum_{i=0}^{m-2}s_{m-i},C_{m-1}+r_m)=C_m.
\end{eqnarray*}

Therefore, the proof of $K_m=C_m$ is completed.

\subsection{$l_m=e_m\ (1\leq m\leq n)$}

We need the following lemma.

\begin{lemma}\label{lem7}
If $S_p+C_p-G_p-A_p> 0$, then $I_p=A_p$.
\end{lemma}
\begin{proof} Induction on $p$.

For $p=1$, by the proof of Lemma \ref{lem6}, we have
$c_1=a_1,s_1>g_1$, so $s_1=h_1=1$, $g_1=r_1=c_1=a_1=0$. Since
$a_1=min(t_1,s_1)=0$ and $s_1=1$, we have $t_1=0$. Then
$i_1=min(t_1,h_1)=0=a_1$.

Assume that for all $1\leq m\le p-1$, $I_m=A_m$ if $S_m+C_m-G_m-A_m> 0$.
Suppose that $S_p+C_p-G_p-A_p> 0$.

If $A_p<C_{p-1}+r_p$, then
$C_p=A_p$, $S_p-G_p>0$, i.e. $G_p=G_{p-1}+g_p<S_p$,
$g_p=r_p<S_p-G_{p-1}$, and so by Lemma \ref{lem2jd}, $C_p=I_p=A_p$.

If $A_p\ge C_{p-1}+r_p $, then $C_p=C_{p-1}+r_p$, and by Lemma
\ref{lem3sb} $g_p=r_p,\ h_p=s_p$. By Lemma \ref{lem6},
$S_{p-1}+C_{p-1}-G_{p-1}-A_{p-1}\ge 0$.  If
$S_{p-1}+C_{p-1}-G_{p-1}-A_{p-1}>0$, by induction, $I_{p-1}=A_{p-1}$
and $ I_p=min(T_p,I_{p-1}+h_p)=min(T_p,A_{p-1}+s_p)=A_p$. If
$S_{p-1}+C_{p-1}-G_{p-1}-A_{p-1}=0$, then $s_p>a_p$, i.e. $s_p=1,\
a_p=0$. Since $a_p=min(T_p-A_{p-1},s_p)$, we have $T_p-A_{p-1}=0$.
Since $a_p=0$, $T_p=A_p$. Since $A_p\le I_p$ and $I_p\le T_p$, we
have $A_p=I_p$.
\end{proof}

\subsubsection{$l_m=e_m,\ 1\leq m\leq n$}

\begin{proof} Induction on $m$.
\begin{eqnarray*}
l_1&=&i_1+g_1-k_1=min(s_1,r_1)+min(t_1,h_1)-k_1 \\
   &=&min(s_1+t_1,s_1+h_1,r_1+t_1,r_1+h_1)-k_1 \\
   &=&min(s_1+t_1,s_1+r_1,r_1+t_1)-k_1, \\
e_1&=&min(b_1,d_1)=min(s_1+t_1-a_1,a_1+r_1-c_1)\\
   &=&min(s_1+t_1+c_1-a_1,a_1+r_1)-c_1 \\
   &=&min(s_1+t_1+c_1-a_1,t_1+r_1,s_1+r_1)-c_1 .
\end{eqnarray*}

If $c_1=a_1$, then $e_1=min(s_1+t_1,s_1+r_1,r_1+t_1)-k_1=l_1$.

If $c_1<a_1$, we have $c_1=0$, $a_1=1$, $s_1=t_1=1,\ r_1=0$, and so
$e_1=l_1=1$.

Therefore $l_1=e_1$.

Assume that for any $1\leq p\le m-1$, $l_p=e_p$.

Since $L_{m-1}=E_{m-1}=I_{m-1}+G_{m-1}-K_{m-1}$ and
$K_{m-1}=C_{m-1}=C_m-c_m$, we have
\begin{eqnarray*}
e_m&=&min(B_m-E_{m-1},d_m)\\
   &=&min(T_m+S_m-A_m-E_{m-1},a_m+r_m-c_m)\\
   &=&min(T_m+S_m-A_m-E_{m-1},min(T_m-A_{m-1},s_m)+r_m-c_m)\\
   &=&min(T_m+S_m-I_{m-1}-G_{m-1}+C_m-A_m,T_m-A_{m-1}+r_m,s_m+r_m)-c_m,   \\
l_m&=&i_m+g_m-k_m\\
   &=&min(T_m-I_{m-1},h_m)+min(S_m-G_{m-1},r_m)-k_m\\
   &=&min(T_m+S_m-I_{m-1}-G_{m-1},T_m-I_{m-1}+r_m,S_m-G_{m-1}+h_m,h_m+r_m)-k_m.
\end{eqnarray*}

There are two cases to consider.

Case 1. $C_m< C_{m-1}+r_m$.
Then $C_m=A_m$.

If $S_m-G_{m-1}<r_m$, then $g_m=s_m=S_m-G_{m-1}=0$, $h_m=r_m=1$, and
by Lemma \ref{lem1dh}, $T_m-A_{m-1}\ge T_m-I_{m-1}\ge 0$. Therefore,
$l_m=min(T_m-I_{m-1},r_m)-k_m=e_m$.

If $S_m-G_m\geq r_m$, then by noting that $C_m<C_{m-1}+r_m$, we have
$c_m=0,\ r_m=g_m=1$ and thus
\begin{eqnarray*}
e_m&=&min(T_m+S_m-I_{m-1}-G_{m-1}, a_m+1)-c_m=T_m+S_m-I_{m-1}-G_{m-1},\\
l_m&=&min (T_m+S_m-I_{m-1}-G_{m-1},T_m-I_{m-1}+1, h_m+1 )-c_m.
\end{eqnarray*}
Since $T_m-I_{m-1}\geq i_m\geq 0$, we have
$l_m=T_m+S_m-I_{m-1}-G_{m-1}=e_m.$

Case 2. $C_m=C_{m-1}+r_m$.

By Lemma \ref{lem3sb}, $g_m=r_m,\ h_m=s_m$. Then
$T_m+S_m-I_{m-1}-G_{m-1}\ge T_m-I_{m-1}+r_m$,
$T_m+S_m-I_{m-1}-G_{m-1}+C_m-A_m=T_m-I_{m-1}+S_{m-1}-G_{m-1}+C_{m-1}-A_{m-1}+s_m+r_m-a_m$ and
$T_m-A_{m-1}+r_m\ge T_m-I_{m-1}+r_m$. By Lemma \ref{lem6},
$S_{m-1}+C_{m-1}-G_{m-1}-A_{m-1}\ge 0$.

If $S_{m-1}+C_{m-1}-G_{m-1}-A_{m-1}>0$, by Lemma \ref{lem7},
$I_{m-1}=A_{m-1}$. Since $s_m-a_m\ge 0$,
$T_m+S_m-I_{m-1}-G_{m-1}+C_m-A_m>T_m-I_{m-1}+r_m$. Therefore,
$l_m=min(T_m-I_{m-1}+r_m,s_m+r_m)-k_m=e_m$.

If $S_{m-1}+C_{m-1}-G_{m-1}-A_{m-1}=0$, then there are two subcases
to consider.

If $s_m=a_m$, $T_m+S_m-I_{m-1}-G_{m-1}+C_m-A_m=T_m-I_{m-1}+r_m$.
Therefore $l_m=min(T_m-I_{m-1}+r_m,s_m+r_m)-k_m=e_m$.

If $s_m>a_m$, then $a_m=i_m=0$, $s_m=1$, $s_m+c_m-g_m-a_m=s_m=1$,
$S_m+C_m-G_m-A_m>0$. Now by Lemma \ref{lem7}, $I_m=A_m$,
$I_{m-1}=A_{m-1}$.
Therefore, $
l_m=min(T_m-I_{m-1},s_m)+r_m-k_m=r_m-k_m=e_m. $
\end{proof}

\begin{lemma}\label{young2}
$FEC=JLK$ is a Young tableau.
\end{lemma}
\begin{proof}
Since $E_p\ge C_p$, $L_p\ge K_p$, $C_p=K_p$ and $E_p=L_p$, we
have $F_p\ge E_p\ge C_p$. By Corollary \ref{youngtableau}, $JLK$ is a
three-column Young tableau.
\end{proof}

\end{document}